\documentclass[reqno]{amsart}
\pdfoutput=1

\usepackage[T1]{fontenc}
\usepackage[utf8]{inputenc}
\usepackage[english]{babel}
\usepackage{tikz}
\usepackage{amssymb,fourier}
\usepackage{braket}
\usepackage{hyperref} 
\usepackage[english]{cleveref}
\usepackage{mathrsfs,enumerate}
\usepackage{faktor}
\usepackage{subcaption}
\newcommand{\refc}{\mathbf}

\newtheorem{theorem}{Theorem}[section]
\newtheorem{lemma}[theorem]{Lemma}
\newtheorem{proposition}[theorem]{Proposition}
\newtheorem{corollary}[theorem]{Corollary}
\theoremstyle{remark}
\newtheorem{remark}[theorem]{Remark}
\theoremstyle{definition}
\newtheorem{example}[theorem]{Example}
\newtheorem{definition}[theorem]{Definition}

\DeclareMathOperator{\Ob}{Ob}
\DeclareMathOperator{\Mor}{Mor}
\DeclareMathOperator{\supp}{supp}
\DeclareMathOperator{\im}{im}

\newcommand{\SSS}{S^1} 
\newcommand{\Z}{\mathbb{Z}} 

\newcommand{\EC}{\mathbb E} 
\newcommand{\AAA}{\mathscr A} 
\newcommand{\WW}{\mathcal{W}} 
\newcommand{\PP}{\mathcal{P}} 
\newcommand{\CCC}{\mathscr C} 
\newcommand{\KK}{\mathcal{K}} 
\newcommand{\FF}{\mathcal F} 
\newcommand{\AAE}{\AAA_{\mathbb{E}}}
\DeclareMathOperator{\codim}{codim}
\DeclareMathOperator*{\colim}{colim}

\newcommand{\interior}[1]{\operatorname{int}(#1)}
\newcommand{\Id}{\operatorname{Id}}
\newcommand{\id}{\operatorname{id}}
\newcommand{\AC}{\texttt{AC}}

\newcommand{\gr}[1]{\lVert {#1} \rVert}

\newcommand{\ie}{i.e.\ }

\usetikzlibrary{decorations.markings}

\tikzset{->-/.style={decoration={
  markings,
  mark=at position 0.7 with {\arrow{>}}},postaction={decorate}}}

\title{The homotopy type of elliptic arrangements}

\author[E. Delucchi]{Emanuele Delucchi}
\address{Emanuele Delucchi \newline {Universit\'e de Fribourg\\D\'epartement de math\'ematiques} \newline Chemin du Mus\'ee 23, CH-1700 Fribourg, Switzerland}
\email{emanuele.delucchi@unifr.ch}
\urladdr{http://www.maestran.ch/math/}

\author[R. Pagaria]{Roberto Pagaria}
\address{Roberto Pagaria \newline {Scuola Normale Superiore} \newline Piazza dei Cavalieri 7, 56126 Pisa, Italy
\newline {Università di Bologna} \newline Piazza di Porta San Donato 5, 40126 Bologna, Italy}
\email{roberto.pagaria@sns.it}
\email{roberto.pagaria@unibo.it}
\urladdr{https://www.dm.unibo.it/~roberto.pagaria/}

\begin{document}

\begin{abstract}
We give combinatorial models for the homotopy type of complements of elliptic arrangements (i.e., certain sets of abelian subvarieties in a product of elliptic curves).  We give a presentation of the fundamental group of such spaces and, as an application, we treat the case of ordered configuration spaces of elliptic curves.

Our models are finite polyhedral CW complexes, and our combinatorial tools of choice are acyclic categories (small categories without loops). 
As a stepping stone, we give a characterization of which acyclic categories arise as face categories of polyhedral CW complexes.
\end{abstract}

\maketitle

\section{Introduction}
\subsection{Background}
The study of arrangements of hypersurfaces in a given ambient space is a classical research topic.
From enumerative problems related to topological dissections to the computation of algebraic-topological invariants of the arrangements' complements, the interaction of combinatorics, topology and geometry is a hallmark of this field. 
 
  This interplay has brought forth interesting results already in the case of arrangements of hyperplanes in complex vector spaces. For instance, the fundamental groups of the complements of such arrangements exhibit interesting structural features that interact subtly with the combinatorial data \cite{infinito}. A prominent role in the advancement of the field has been played by combinatorial models for the complements. In particular, Salvetti used the polyhedral data of an arrangement of hyperplanes in Euclidean space in order to construct a simplicial complex that carries the homotopy type of the complexified arrangements'complement \cite{salvetti}. Bj\"orner and Ziegler have studied a class of simplicial models for complements of general complex arrangements \cite{BiZi}.

 Other classes of arrangements, beyond the case of hyperplanes, came into the focus of a growing amount of research over the last decade. Arrangements in tori, or ``toric arrangements'', gave rise to new combinatorial structures \cite{MB,MdA,FiMo,GAOS} as well as to an intensive study of the topology and geometry of the complement of arrangements in the complex torus \cite{a5,DD2,DCG1,DCG2,DCG3,DCP05,MociPagaria}.
Also in this ``toric'' case, a recurring and interesting theme has been the study of the interplay between combinatorics and topology \cite{Pagaria1,Pagaria2}.

 Several aspects of the theory of arrangements of hyperplanes have been generalized to toric arrangements and beyond, pointing to the potential for a general  treatment of the subject.

\subsection{Subject and results of this paper}
In this paper we focus on arrangements of special subvarieties in products of elliptic curves called {\em elliptic arrangements} (see \Cref{arrangements} for the precise definition and \Cref{generality} for a discussion of our level of generality compared to existing literature).  These arrangements were first considered by Varchenko and Levin \cite{VL12} and, together with arrangements in complex vector spaces and in the complex torus, they have been in the focus of intense study in the last decade. 
However,  the topology of elliptic arrangements is by far the less understood in this triad.
For instance, despite the availability of algebraic models for the cohomology \cite{Bibby16,Dupont} and some results on their dependency on the combinatorial data \cite{PagariaBrutto},  basic topological invariants -- e.g., Betti numbers  -- remain out of reach \cite{PagariaContrEll} even for very special cases such as the configuration space of $n$ points in an elliptic curve, which can be realized as the complement space of an elliptic arrangement. 

\begin{itemize}
\item We construct a combinatorial finite cell complex that carries the homotopy type of the arrangements' complement. The combinatorial tool of choice are finite acyclic categories. These discrete structures have emerged as a flexible and powerful tool, e.g., in computational topology \cite{Kozlov}, metric geometry \cite{Bridson-Haefliger} and beyond. 
Thus our model  puts the elusive topology of elliptic arrangements in the range of the substantial computational and structural toolbox of combinatorial algebraic topology.
\item We use our model to give a presentation of the complements' fundamental group by generators and relations. As an application, we give a new combinatorial presentation of the fundamental group of the configuration space of $n$ points in an elliptic curve.
\end{itemize}
 The combinatorial data needed for our construction is the {\em face category} of an associated toric arrangement. This is a finite category without loops that encodes the induced cellularization of the compact torus. Such categories have demonstrated their usefulness in a variety of contexts, not least in proving results on the topology of toric arrangements \cite{DD2}. Accordingly, in the extant literature they are treated from different points of view, see e.g.\ \cite{Bridson-Haefliger,Kozlov}. 
 \begin{itemize}
 \item We review and partly develop the combinatorial theory of nerves of face categories of polyhedral complexes with group actions, as a generalization of Bj\"orner's poset-theoretic approach to the study of regular CW-complexes \cite{BjornerCW}. 
 \end{itemize}

A generalization of our results to wider classes of  arrangements is possible using the toolkit  we develop in \Cref{appendix}. We leave this line of research open, but suggest  a concrete approach to models for a class of arrangements in abelian Lie groups studied by Liu, Tran and Yoshinaga (see \Cref{lty}) as well as for the fullest generality of elliptic arrangements (see \Cref{approach}).

\subsection{Structure of the paper}  In \Cref{arrangements} we define and contextualize elliptic arrangements. \Cref{models} contains the construction of the combinatorial model and the proof that it carries the correct homotopy type. A finite presentation for the complement's fundamental group is given in \Cref{sect:fundamental_group}. As an application, in \Cref{An} we discuss elliptic arrangements of Coxeter type $A$, whose complements are the configuration spaces of of points in elliptic curves, and we give the explicit form of our presentation for the fundamental group in this case. The necessary notions about topology and combinatorics of polyhedral complexes and their face categories are gathered, resp.\ proved in the \Cref{appendix}.

\subsection{Acknowledgements} 
The first author was supported by the Swiss National Science Foundation professorship grant PP00P2\_150552/1. This research was initiated during a visit of the second author at the university of Fribourg and was completed during first author's stay as an ISA fellow at the University of Bologna. Both authors thank these institutions for for the friendly and supportive hospitality.

\section{Arrangements}\label{arrangements}
Our basic data is an elliptic curve $\mathbb E=\mathbb C/\langle 1, \omega \rangle$ (for a fixed $\omega\in \mathbb C\setminus \mathbb R$), a full-row-rank integer matrix $A\in \mathbb Z^{d\times n}$ whose columns we label $a_1,\ldots,a_n$, and some numbers $b_1,\ldots,b_n\in \mathbb R/\mathbb Z$. For each $j=1,\ldots,n$ we have hypersurfaces
\[
\mathbf{H}_j:=\left\{ z\in \mathbb E^d \,\left\vert\, \sum_k z_k a_{j,k} = b_j+ \omega b_j \right.\right\}\subseteq \mathbb E^d,
\]
\[
{H}_j:=\left\{ z\in \mathbb (S^1)^d \,\left\vert\, \prod_k z_k^{a_{j,k}} = b_j \right.\right\}\subseteq (S^1)^d.
\]
The {\em toric arrangement} defined by $A$ and $b$ is
\[
\AAA:=\set{H_1,\ldots,H_n}.
\]
The {\em elliptic arrangement} associated to $A$ and $b$ is 
\[
\AAE=\set{\mathbf{H}_1,\ldots,\mathbf{H}_n}.
\]
We will be interested in the topology of the complement of an elliptic arrangement, \ie the space
\[
M(\AAE):=\EC^d\setminus \cup \AAE.
\]

\begin{remark}
Our requirement that the matrix $A$ is full-row-rank is sometimes referred to as the arrangement being {\em essential}, i.e., the minimal intersections have dimension $0$. If $\AAE$ is a non-essential arrangement, then $\mathbb{E}^d$ splits as a product $\mathbb{E}^{k}\times \mathbb{E}^{d-k}$, where $k$ is the dimension of any minimal intersection of $\AAE$ and $A$ defines an essential arrangement $\overline{\AAE}$ in $\mathbb{E}^{d-k}$. Moreover, $M(\AAE)=\mathbb{E}^{k}\times M(\overline{\AAE})$ and our methods can be applied in order to model the homotopy type of $M(\overline{\AAE})$.
\end{remark}

\begin{example}[The arrangements of Coxeter type $A_d$]\label{es:1}
The arrangement of reflecting hyperplanes for the standard representation of the Coxeter group of type $A_d$ is given by the hyperplanes of equation 
 \begin{equation}\label{ciliegina}
 x_i=x_j \quad\textrm{ for }\quad 0\leq i < j \leq d
 \end{equation} and is customarily referred to as the ``braid arrangement''.
It is not essential since all the hyperplanes contain the line $\ell:\,x_0=x_1= \dots = x_d$. 
One can obtain an essential arrangement with the same combinatorics and the same homotopy type by intersecting the braid arrangement with any complement of the line $\ell$. A frequent choice for such complement is the subspace $\{x_0+x_1+\dots +x_d=0\}$.
In the toric and elliptic case, the arrangement defined by the equations \eqref{ciliegina} is again not essential.
However, now we cannot perform the same choice of complement to $\ell$ because $\{x_0+x_1+\dots +x_d=0\} \subset \Z^{d+1}$ and $\ell\cap \Z^{d+1}=\Z(1,1,\dots,1)$ together do not span $\Z^{d+1}$. 
Our choice is to take $\{x_0=0\} \subset \Z^{d+1}$ as complement of the line $\Z(1,1,\dots,1)$. This  is motivated by the application to configuration spaces, see \Cref{rem:found_group}.
We call this essential arrangement the {\em toric (resp.\ elliptic) arrangement of Coxeter type $A_d$}.
Explicitly, such arrangements contain one hypersurface $H_{0,i}$ with equation $x_i=0$ for all $i=1,\ldots,d$ and one hypersurface $H_{i,j}$ with equation $x_i=x_j$ for all $1\leq i < j \leq d$.  
A picture of the case $d=2$ is displayed in Figure \ref{fig:A2_in_S1}.
Our definition follows \cite{BibbyRepr,DPG,MociRoot}, while we remark that the toric arrangements of Coxeter type discussed by Aguilar and Petersen \cite{AgPe} are "finite covers" of those we define. More precisely: if we let $\mathscr A^\upharpoonright$ denote the infinite arrangement in $\mathbb R^d$ given by the reflection hyperplanes of the affine group $\widetilde{A_d}$, Aguiar and Petersen consider the induced arrangement in the torus obtained as the quotient of $\mathbb R^d$ by the coroot lattice of $A_d$. Under this point of view, our definition corresponds to quotienting by the full weight lattice. It is in this sense that the two definitions differ by a covering map whose degree is the index of the coroot lattice in the weight lattice, i.e., $d+1$.
\end{example}

\begin{remark}\label{generality}
Elliptic arrangements were defined for the first time in \cite{VL12} by Levin and Varchenko. Their definition 
coincides with ours, although they only consider \emph{central} elliptic arrangements, \ie $c_j =0$ for all $j$.
Suciu in \cite{Suciu16} has studied a larger class of elliptic arrangements,  allowing translations by any element $c_j \in \EC$ (we consider only translations with $c_j= b_j + \omega b_j$ for some $b_j \in S^1$).\\
It is natural to extend the definition of an elliptic arrangement by allowing for divisors of the form $H=\varphi^{-1}(c)$ where $\varphi$ is any group homomorphism  $\mathbb E^n \to \mathbb E$ and $c$ is any element of $\mathbb E$. When $\mathbb E$ is not complex multiplication, then any such $\varphi$ is defined by an element of $\mathbb Z^n$ as in our setup. Otherwise, $\operatorname{End}\mathbb E \supsetneq \mathbb Z$ and thus there are  homomorphisms $\mathbb E^n \to \mathbb E$ that cannot be obtained from integer vectors. 
 The only results available  in this case were obtained for general {\em Abelian arrangements} (arrangements of abelian subvarieties in an abelian variety) by Bibby \cite{Bibby16} and, in the even greater generality arrangements of hypersurfaces, by Dupont \cite{Dupont}.\\
 The above-mentioned  ``non-integral'' elliptic arrangements  are not in the scope of this paper.
In particular, we leave the construction of cellular models for general elliptic arrangements as an open problem (see also \Cref{approach}). 
\end{remark}

\section{A combinatorial model for elliptic arrangements}\label{models}
The lift of a toric arrangement $\AAA$ through the universal cover of the torus $(S^1)^d$ is a periodic arrangement of hyperplanes $\AAA^\upharpoonright$ in $\mathbb R^d$. This periodic arrangement defines a structure of polyhedral CW complex $K_{\AAA^\upharpoonright}$ on $\mathbb R^d$.
The deck transformation group is $\mathbb Z^d$ acting freely on $\mathbb R^d$ and cellularly on $K_{\AAA^\upharpoonright}$. We call $K_\AAA$ the quotient complex $K_{\AAA^\upharpoonright}/\mathbb Z^d$. This is a polyhedral CW complex structure on the space $(S^{1})^d$. See \Cref{appendix} for terminology and definitions on polyhedral complexes and face categories.

\begin{definition}\label{df:FA}
Given a toric arrangement $\AAA$ let $\mathscr F(\AAA):=\FF(K_{\AAA})$ be the face category of the complex $K_\AAA$.
\end{definition}

\begin{remark}\label{rmk:quot_face_cat}
In particular, notice that  $\FF(\AAA)\simeq \FF(K_{\AAA^{\upharpoonright}})/\Z^d$.
\end{remark}

In this section we keep the notations established in \Cref{arrangements}. Moreover, when the context allows it, we will write $\mathscr F$ as a shorthand for $\mathscr F(\AAA)$.

For every $F\in \Ob(\FF)$ let $$\supp(F):=\set{H\in \AAA \mid F\subseteq H}$$ denote the set of all hypersurfaces in $\AAA$ that contain $F$.

The product $K_\AAA\times K_\AAA$ is the space $(S^1)^{d+d}$, which we consider with the natural product cellularization. This polyhedral CW complex has face category equal to $\FF\times \FF$. (\cite[Theorem A.6]{HatcherBook}).

\begin{definition}
\label{def:KK}
Define a category $\mathcal K$ as follows.
\[\Ob(\mathcal K):=\set{(F,G)\in \FF \times \FF \mid 
\supp(G)\cap \supp(F)=\emptyset 
}\]
\[\Mor_{\KK}((F_1,G_1),(F_2,G_2)):=
\set{
(\alpha,\beta) \mid \alpha \in \Mor_{\FF^{op}}(F_1,F_2), \beta \in \Mor_{\FF^{op}}(G_1,G_2)
}\]
\end{definition}

\begin{remark}
The category $\mathcal{K}^{op}$ is a full subcategory of the product $\FF \times \FF$ (see \cite[§II.3]{maclane}). 
\end{remark}

\begin{lemma} \label{lemma:phi}
For any elliptic arrangement $\AAE$, the total space $\cup\AAE$ is mapped  homeomorphically to the subcomplex of $K_\AAA\times K_\AAA$ associated to the full subcategory $\mathscr D$ of $\FF\times\FF$ defined  on the object set 
\[\Ob\mathscr D=\set{(F,G)\in \Ob(\FF\times\FF) \mid \supp(F)\cap \supp(G)\neq \emptyset },\]
giving
\[
\cup\AAE \simeq \gr{\mathscr D} \subseteq \gr{\FF\times \FF}.
\]
\end{lemma}

\begin{remark}\label{remop}
Notice that, writing $X:=\Ob\mathscr D$, in the language of \Cref{lem:retract} we can write $\mathscr D =(\FF\times\FF)[X]$, $\KK^{op}=(\FF\times\FF)[X^c]$.
\end{remark}

\begin{proof}[Proof of \Cref{lemma:phi}]
Let as above $\EC=\mathbb C/\langle 1, \omega \rangle$ and consider the map
\[
\varphi \colon \EC^d \to (S^1)^{d+d},\quad\quad z\mapsto (\Re_{\omega}(z),\Im_{\omega}(z)),
\]
where for every $z\in \mathbb E$ we consider the unique expression $z=x+\omega y$ with $x,y\in \mathbb R/\mathbb Z$, and let $\Re_{\omega}(z):=x$, $\Im(z)_{\omega}:=y$.
This is a homeomorphism that maps $\cup\AAE$ to a subcomplex of $K_\AAA\times K_\AAA$. In fact, for every $\mathbf{H}_j\in \AAE$ defined by an integer vector $a_j\in \mathbb Z^d$ and a value $b_j\in S^1$, a $z\in \mathbb E^d$ is in $\mathbf{H}_j$ if and only if $a_j^Tz=a_j^T\Re(z)+\omega a_j^T\Im(z) = b_j+\omega b_j$, \ie $a_j^T\Re(z)=a_j^T\Im(z)=b_j$.
Thus, for any $z\in \EC^d$, 
\[
\mathbf{H}_j\in\supp_{\AAE}(z)
\quad\Leftrightarrow\quad
H_j\in\supp_{\AAA}\Re(z) \cap \supp_{\AAA}\Im(z). \qedhere
\]
\end{proof}

\begin{example}
Consider the elliptic arrangement $A_2$ described by the equations
\begin{align*}
& z_1 = 0, \\
& z_2 = 0, \\
& z_1+z_2 = 0.
\end{align*}
\begin{figure}
\begin{subfigure}{0.49\textwidth}
\centering
\begin{tikzpicture}
		\draw[gray] (0, 0) rectangle (2,2);
		\draw [color=blue,line width=1.5pt] (0,0) -- node[label={[shift={(0,-0.7)}]$b$}]{} (2,0);
		\draw [color=green,line width=1.5pt] (0,0) -- node[label={$c$}]{} (2,2);
		\draw [color=red,line width=1.5pt] (0,0) -- node[label={[shift={(-0.2,0)}]$a$}]{} (0,2);
		\draw (0,0) node [circle, draw, fill=black, inner sep=0pt, minimum width=4pt,label={[shift={(-0.2,0)}]$p$}] {};
		\node (c1) at (0.5,1.3) {$C_1$};
		\node (c2) at (1.3,0.5) {$C_2$};
\end{tikzpicture}
\caption{The $A_2$ arrangement in the unit square.}\label{fig:A2_in_S1}
\end{subfigure}
\begin{subfigure}{0.49\textwidth}
\centering
\begin{tikzpicture}
    \node (p) at (0,2) {$p$};
    \node (e1) at (-1,1)  {$a$};
    \node (e2) at (0,1) {$b$};
    \node (e3) at (1,1) {$c$};
    \node (f1) at (-0.5,0) {$C_1$};
    \node (f2) at (0.5,0) {$C_2$};

    \draw [->] (p) to [bend left=20] (e1);
    \draw [->] (p) to [bend left=20] (e2);
    \draw [->] (p) to [bend left=20] (e3);
    \draw [->] (p) to [bend right=20] (e1);
    \draw [->] (p) to [bend right=20] (e2);
    \draw [->] (p) to [bend right=20] (e3);
    \draw [->] (e1) to (f1);
    \draw [->] (e2) to (f1);
    \draw [->] (e3) to (f1);
    \draw [->] (e1) to (f2);
    \draw [->] (e2) to (f2);
    \draw [->] (e3) to (f2);
\end{tikzpicture}
\caption{The face category $\FF(K_{A_2})$.}
\label{fig:cat_F_A2}
\end{subfigure}
\caption{The arrangement $A_2$ in $\SSS$.}
\end{figure}
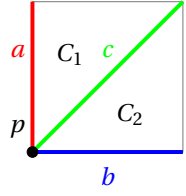
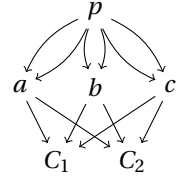
In \Cref{fig:A2_in_S1}, the corresponding arrangement in $S^1$ is represented.
The Hasse diagram of the face category $\FF(K_{A_2})$ is draw in \Cref{fig:cat_F_A2}.
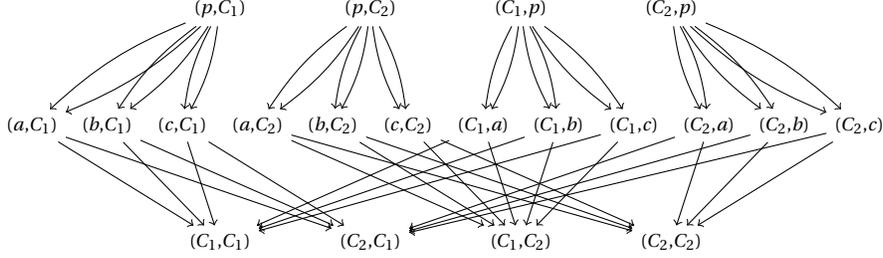
\begin{figure}
\begin{tikzpicture}[y=5em]
\foreach \x in {1,2}{
    \node (p \x) at (-5+2*\x,2) {$\scriptstyle (p,C_{\x})$};
    \node (\x p) at (-1+2*\x,2) {$\scriptstyle (C_\x,p)$};
    \foreach \y/\ytext in {1/a,2/b,3/c}{
		\node (\ytext \x) at (-9.5+\y+3*\x,1)  {$\scriptstyle (\ytext,C_\x)$};
    \node (\x \ytext) at (-3.5+\y+3*\x,1)  {$\scriptstyle (C_\x,\ytext)$};}
    \foreach \y in {1,2}{
		\node (f\x \y) at (-9+2*\x+4*\y,0) {$\scriptstyle (C_\x,C_\y)$};
    }}

\foreach \x in {1,2}{
	\foreach \ytext in {a,b,c}{
    		\draw [->] (p \x) to [bend left=10] (\ytext \x);
    		\draw [->] (p \x) to [bend right=10] (\ytext \x);
		\draw [->] (\x p) to [bend left=10] (\x \ytext);
    		\draw [->] (\x p) to [bend right=10] (\x \ytext);
    		\foreach \z in {1,2}{
    			\draw [->] (\x \ytext) to  (f\x \z);
    			\draw [->] (\ytext \x) to  (f\z \x);
    }}}
\end{tikzpicture}
\caption{The category $\KK(A_2)$.}
\label{fig:cat_K_A2}
\end{figure}
The category $\KK_{A_2}$ is represented in \Cref{fig:cat_K_A2}.
\end{example}

\begin{theorem}
Let $\AAE$ be an elliptic arrangement and $\mathcal{K}$ be the associated category.
We have a homotopy equivalence
\[
\lVert \mathcal K \rVert \simeq M(\AAE).
\]
\end{theorem}
\begin{proof}
By \Cref{lemma:phi}, the space $M(\AAE)$ is homeomorphic, via $\varphi$, to the complement of $\lVert \mathscr D \rVert$ inside $\lVert \FF\times \FF \rVert$.
With \Cref{lem:retract} and \Cref{remop}, this space then retracts onto its subcomplex $\lVert \mathcal K^{op} \rVert=\lVert \mathcal K \rVert$. 
\end{proof}

We now prove that $\KK$ is the face category of a polyhedral CW complex, thus giving a cellularization of the space $\gr{\KK}$ that is much more economical than the triangulation coming from the definition of the nerve. We will use this for example in presenting the fundamental group. 
In order to apply the criterion given in \Cref{lem:sphere} we need to examine slice categories of $\KK$ (see \Cref{def:slice} and  \Cref{appendix} for terminology).

\begin{lemma}\label{lem:sphere_ellittico}
For any given $(F,G)\in \Ob\mathscr K$, the slice category $\mathscr K / (F,G)$ is equivalent to the face poset of a $(\codim(F)+\codim(G))$-dimensional polytope.
\end{lemma}

\begin{proof}
Let $\FF^{\upharpoonright}$ denote the face category (in fact, a poset) of the polyhedral complex $K_{\AAA^\upharpoonright}$ determined by the affine, essential hyperplane arrangement $\AAA^{\upharpoonright}$ in $\mathbb R^d$. For every face $P\in \Ob \FF^{\upharpoonright}$, the poset $(\FF^{\upharpoonright})_{\geq P}$ is order dual to the poset of faces of a polytope $Z_P$ of dimension $\codim(P)$ (the {zonotope} of the arrangement of all hyperplanes containing $P$) \cite[Corollary 7.17]{Ziegler}. In particular, with \Cref{lem:CW} we have that $(\FF^\upharpoonright)^{op}$ is the face category of a polyhedral CW complex.
Recall that $K_{\AAA}$ is the quotient of $K_{\AAA^\upharpoonright}$ by a free cellular action, and $\FF$ is the quotient by the induced free action on $\FF^{\upharpoonright}$, see \Cref{rmk:quot_face_cat}. Passing to opposites, $\FF^{op}$ is the quotient by a free action on $(\FF^\upharpoonright)^{op}$, hence it is covered by the latter category in the sense of \Cref{lem:cover_cat}. In particular, $\FF^{op}/F$ is isomorphic to $(\FF^\upharpoonright)^{op}/F'$ for any lift $F'$ of $F$, and hence to the face poset of the polytope $Z_{F'}$. 
Analogously, we find a $(\codim(G))$-dimensional polytope $Z_{G'}$ such that $\mathscr F^{op} /G \simeq \mathscr F(Z_{G'})$. \\
Now, by definition ${\mathscr K / (F,G)}\simeq{\mathscr F^{op} /F} \times {\mathscr F^{op} /G}$ and with the preceding discussion we can write 
$$
{\mathscr K / (F,G)}\simeq \FF(Z_{F'})\times \FF(Z_{G'})\simeq \FF(Z_{F'}\times Z_{G'})
$$
where the product $Z_{F'}\times Z_{G'}$ is a polytope of dimension $\dim(Z_{F'})+\dim(Z_{G'})= \codim (F) +\codim(G)$, see \cite[pp.\ 9-10]{Ziegler}.
\end{proof}

\begin{proposition}
$\mathcal K$ is the face category of a polyhedral complex $K({\AAE})$.
\end{proposition}
\begin{proof}
The claim follows from \Cref{lem:sphere_ellittico} and \Cref{lem:CW}. 
\end{proof}

\begin{remark}\label{lty} We point out that along the same lines (i.e, via \Cref{lem:retract,lem:CW}) it is possible to describe an acyclic category and a polyhedral complex with the homotopy type of the complement of any of the  arrangements in $((S^1)^p\times \mathbb R^q)$ considered by Liu, Tran and Yoshinaga \cite{LiuTranYoshinaga}, generalizing known models for complexified hyperplane arrangements \cite{salvetti}, generalized configuration spaces \cite{Salvetti-Mori}, and toric arrangements \cite{DD1}. We do not pursue this line any further here.
\end{remark}

\begin{remark}\label{approach}
When the hypersurfaces in $\AAA_{\mathbb E}$ are  defined by equations of the form $\sum_k a_{k}z_k = c$ with $a_k\in \operatorname{End}\mathbb E\subseteq \mathbb C$ for all $k$ and $c\in \mathbb E$ arbitrary (see \Cref{generality}), the arrangement $\AAA_{\mathbb E}$ can still be lifted to an arrangement of (affine, non-complexified) hyperplanes in the universal cover $\mathbb C^d$ of $\mathbb E^d$. A cellular model for $M(\AAA_{\mathbb E})$ can be obtained as a quotient of the  Bj\"orner-Ziegler model \cite{BiZi} for the lifted arrangement,  using the tools outlined in the Appendix. We leave this line of thought for further research.
\end{remark}

\section{The fundamental group}
\label{sect:fundamental_group}

If $\CCC$ is the face category of a polyhedral CW complex $K$ we let $\CCC^i$ be the set of all objects of $\CCC$ corresponding to faces of $K$ of dimension $i$.
We continue writing simply $\FF$ for the face category $\FF(\AAA)$ of the rank $d$ essential toric arrangement $\AAA$ associated to the elliptic arrangement $\AAA_{\mathbb{E}}$. 

Moreover, we borrow some standard terminology from hyperplane arrangements and call {\em chambers}, resp.\ {\em walls} the elements of $\FF^d$, resp.\ $\FF^{d-1}$

Since every wall is in the boundary of two chambers (or of one chamber in two different ways), we can define the graph $G_\AAA$ whose vertices are the chambers and whose edges (or loops) are the walls.
We fix an orientation of $G_\AAA$ and choose a rooted spanning tree $T \subseteq  \FF^{d-1}$ in the connected graph $G_\AAA$.
Call $C_0$ the chamber corresponding to the root of the tree.

We define a tree $T_\KK$ in the $1$-skeleton of the polyhedral CW complex $K({\AAE})$ whose set of edges is 
\[ T_\KK  := \set{(t,C_0)\}_{t \in T} \sqcup \{(C,t)}_{t \in T, C\in \FF^d},\]
where we denote cells of $K({\AAE})$ (\ie objects of $\KK\subseteq \FF^2$) by pairs of objects in $\FF$.

\begin{lemma}
\label{lemma:tree_T_K}
The subgraph $T_\KK$ is a spanning tree in the $1$-skeleton of $K({\AAE})$.
\end{lemma}
\begin{proof}
Let $f_d$  be the cardinality of $\FF^d$.
The cardinality of $\KK^0$ is $f_d^2$ and the cardinality of $T_\KK$ is $f_d^2-1$.   The graph induced by $T_\KK$ is easily seen to be connected and, thus, it is a spanning tree in the $1$-skeleton of $K({\AAE})$.
\end{proof}

Choose an arbitrary orientation of every edge of the graph $G_\AAA$ defined above. This induces an orientation of the edges in the $1$-skeleton of $K({\AAE})$. Indeed, such edges are pairs $(C,a)$ or $(a,C)$, where $a$ is an edge of $G_\AAA$ and $C\in \FF^{d}$. If we write $A_1,A_2$ for the tail, resp.\ head of $a$ in $G_\AAA$, then we may consider $(a,C)$ as being oriented from $(A_1,C)$ to $(A_2,C)$, and analogously $(C,a)$ from $(C,A_1)$ to $(C,A_2)$.  

A set of generators for the fundamental group $\pi_1(M(\AAA_{\mathbb{E}}))$ is given by 
the closed walks $\gamma_e$ in $T_\KK$, one for each oriented edge $e \in \KK^1 \setminus T_\KK$, where $\gamma_e$ is the concatenation of the path in $T_\KK$ from $(C_0,C_0)$ to the tail end of $e$, $e$ itself, and the path in $T_\KK$ from the head of $e$ to $(C_0,C_0)$. 
This generators can be grouped into three types:
\begin{enumerate}[(i)]
\item $(a,C_0)$ for $a\in \FF^{d-1}(\AAA) \setminus T$,
\item $(a,C_i)$ for $i \neq 0$ and $a\in \FF^{d-1}(\AAA)$,
\item $(C_i, a)$ for $a\in \FF^{d-1}(\AAA) \setminus T$ and all $i$.
\end{enumerate}

The relations are given by the elements in $\KK^2$, i.e., all pairs of the form $(p,C_i)$, $(C_i, p)$ and $(a,b)$ with $p\in \FF^{d-2}$, $a,b\in \FF^{d-1}$ such that $\supp(a) \cap \supp(b) = \emptyset$.

In the following, we set the element $(a,C_i)=1$ if it does not belong to our set of generators.
The element $(a,b)$ gives the relation
\begin{equation}\label{rel1}
(A_0,b)(a,B_1)=(a,B_0)(A_1,b),
\end{equation}
where $a$ is oriented from $A_0$ to $A_1$ and $b$ from $B_0$ to $B_1$, see \Cref{fig:rel1}.

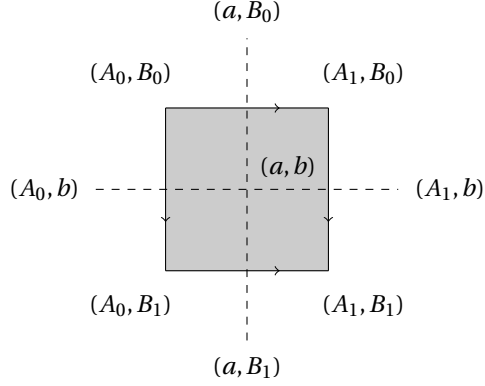
\begin{figure}
\centering
\begin{tikzpicture}[x=5em, y=5em]
\node (Q1) at (1,1) {$(A_1,B_0)$};
\node (Q2) at (-1,1) {$(A_0,B_0)$};
\node (Q3) at (-1,-1) {$(A_0,B_1)$};
\node (Q4) at (1,-1) {$(A_1,B_1)$};
\draw[dashed] (-1.3,0) node[label={180:$(A_0,b)$}]{} -- (1.3,0) node[label={0:$(A_1,b)$}]{};
\draw[dashed] (0,-1.3) node[label={-90:$(a,B_1)$}]{} -- (0,1.3) node[label={90:$(a,B_0)$}]{};
\coordinate (P1) at (-0.7,0.7);
\coordinate (P2) at (0.7,0.7);
\coordinate (P3) at (0.7,-0.7);
\coordinate (P4) at (-0.7,-0.7);
\draw[->-] (P1) -- (P2); 
\draw[->-] (P2) -- (P3);
\draw[->-] (P1) -- (P4); 
\draw[->-] (P4) -- (P3);
\draw[fill=black, opacity=0.2] (P1) -- (P2) -- (P3) -- (P4);
\node [circle, inner sep=0pt, minimum size=4pt, label={45:$(a,b)$}] (O) at (0,0) {};
\end{tikzpicture}
\caption{The cells of $\KK$ contained in $(a,b)$}
\label{fig:rel1}
\end{figure}

The relations for the elements $(p,C)$ and $(C, p)$ are analogous to those in the presentation of the fundamental groups of complements of complexified  hyperplane arrangements given, e.g, in \cite{salvetti}.
Let $a_1, a_2, \dots, a_k$ be a counter-clockwise ordering of the walls containing $p$.
Set $n_j$ equal $1$ if $a_j$ is oriented counter-clockwise and equal $-1$ otherwise.
The relations given by $(p,C)$ and $(C,p)$ are, respectively,
\begin{equation}\label{rel2}
 \prod_{j=1}^k (a_j,C)^{n_j} = 1\quad
 (\textrm{from } (p,C)), \quad\quad\quad  
 \prod_{j=1}^k (C,a_j)^{n_j} = 1\quad
 (\textrm{from } (C,p)),
\end{equation}
as shown in \Cref{fig:rel2}.

\begin{figure}
\centering
\begin{tikzpicture}[x=6em, y=6em]
\node [circle, inner sep=0pt, minimum size=4pt, label={-20:$(p,C)$}] (O) at (0.1,0) {};
\draw[dashed] (-1,0) node[label={180:$(a_1,C)$}]{} -- (1,0) node[label={0:$(a_{i},C)$}]{};
\draw[dashed] (-.8,.8) node[label={135:$(a_2,C)$}]{} -- (.8,-.8) node[label={-45:$(a_{i+1},C)$}]{};
\draw[dashed] (0,1) node[label={90:$(a_3,C)$}]{} -- (0,-1) node[label={-90:$(a_{i+2},C)$}]{};
\draw[fill=black,opacity=0.2] (-0.7,0.3) -- (-0.3,0.7) -- (-0.3,0.7) -- (0.3,0.7) -- (0.3,0.7) -- (0.7,0.3) -- (0.7,0.3) -- (0.7,-0.3) -- (0.7,-0.3) -- (0.3,-0.7) -- (0.3,-0.7) -- (-0.3,-0.7) -- (-0.3,-0.7) -- (-0.7,-0.3) --  (-0.7,-0.3) -- (-0.7,0.3);
\draw[->-] (-0.7,0.3) -- (-0.3,0.7);
\draw[->-] (-0.3,0.7) -- (0.3,0.7);
\draw[->-] (-0.7,-0.3) -- (-0.7,0.3);
\draw[->-] (0.3,-0.7) -- (-0.3,-0.7);
\draw[->-] (0.7,-0.3) -- (0.3,-0.7);
\draw[->-] (0.7,0.3) -- (0.7,-0.3);
\draw[dotted] (0.3,0.7) -- (0.7,0.3);
\draw[dotted] (-0.3,-0.7) -- (-0.7,-0.3);
\draw[dashed] (0.6,.9) -- (-0.6,-.9);
\draw[dashed] (.9,0.6) -- (-.9,-0.6);
\draw[dotted] (0.6,0.9) to [bend left=20] (0.9,0.6);
\draw[dotted] (-0.6,-0.9) to [bend left=20] (-0.9,-0.6);
\end{tikzpicture}
\caption{The cells of $\KK$ contained in $(p,C)$}
\label{fig:rel2}
\end{figure}
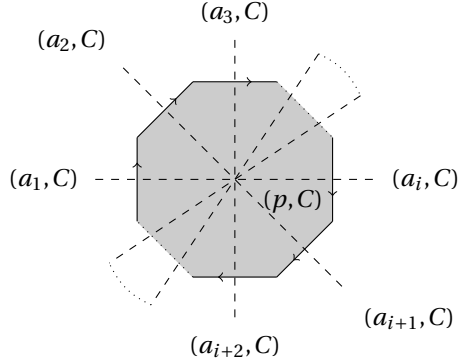

\bigskip

We summarize the discussion of this section as follows.

\begin{theorem} \label{thm:pres}
A presentation for the fundamental group of the complement of an elliptic arrangement is given by generators of the form (i), (ii), (iii) above, with relations of type \eqref{rel1}, \eqref{rel2} above. \hfill\qedsymbol
\end{theorem}

\section{Example: The Coxeter case \texorpdfstring{$A_n$}{An}}
\label{An}

\subsection{Cyclic partitions}
\begin{definition}
A \textit{cyclic partition} of $n$ is a partition of the set $\set{0,1, \dots,n}$ whose set of  blocks is cyclically ordered.
\end{definition}

We represent a cyclic partition with $k$ blocks by cyclic strings with $k$ vertical separators, which we write in line by thinking of the cyclic string being "cut'' immediately after a separator. This way we will identify two "linearized" strings if they differ by a cyclic permutation.

\begin{example}
The cyclic partitions of $2$ can be represented as follows: 
\begin{align*}
& 0|1|2| \;\; (\simeq 1|2|0| \simeq 2|0|1|) & & 0|2|1| & & 012| \\
& 01|2| & & 02|1| & & 0|12|.
\end{align*}
\end{example}

The \textit{category of cyclic partitions} of $n$ is the acyclic category (see \Cref{app:cat}) whose objects are the cyclic partitions of $n$ and the morphisms are defined in the following way. 
The set $\Mor (x,y)$ is non-empty if and only if each block of $x$ is the union of  adjacent blocks of $y$. If $x$ has $k$ blocks and $y$ has $h$ blocks, the elements of $\Mor(x,y)$ are the different sets of $h-k$ separators in $y$ whose removal from $y$ leaves $x$. Composing two morphisms corresponds to taking union of the separator sets.
Notice that $\lvert \Mor (x,y) \rvert \leq 1$ for $x \neq 01\dots n|$. Moreover, $\lvert \Mor (01\dots n|,y) \rvert$ is equal to the number of blocks in $y$.

\begin{example}
The category of cyclic partitions of $2$ is depicted in \Cref{fig:cat_faces_A2}.
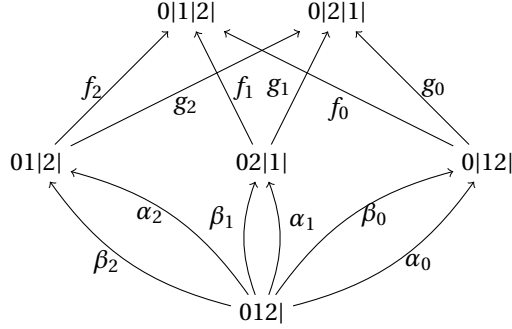
\begin{figure}
\centering
\begin{tikzpicture}
    \node (p) at (0,0) {$012|$};
    \node (e1) at (-3,2)  {$01|2|$};
    \node (e2) at (0,2) {$02|1|$};
    \node (e3) at (3,2) {$0|12|$};
    \node (f1) at (-1,4) {$0|1|2|$};
    \node (f2) at (1,4) {$0|2|1|$};

    \draw [->] (p) to [bend left=20] node [left] {$\beta_2$} (e1);
    \draw [->] (p) to [bend left=20] node [above left] {$\beta_1$} (e2);
    \draw [->] (p) to [bend left=20] node [right] {$\beta_0$} (e3);
    \draw [->] (p) to [bend right=20] node [left] {$\alpha_2$} (e1);
    \draw [->] (p) to [bend right=20] node [above right] {$\alpha_1$} (e2);
    \draw [->] (p) to [bend right=20] node [right] {$\alpha_0$} (e3);
    \draw [->] (e1) to node [left] {$f_2$} (f1);
    \draw [->] (e2) to node [right] {$f_1$} (f1);
    \draw [->] (e3) to node [below] {$f_0$} (f1);
    \draw [->] (e1) to node [below] {$g_2$} (f2);
    \draw [->] (e2) to node [left] {$g_1$} (f2);
    \draw [->] (e3) to node [right] {$g_0$} (f2);
\end{tikzpicture}
\caption{The category of cyclic partition of $2$.}
\label{fig:cat_faces_A2}
\end{figure}
Notice that the maps $\alpha_i$ corresponds to removal of the separator before $i$ (i.e., removing ``$\ast |i$''), $\beta_i$ to removal of  the separator after $i$ (i.e., removing ``$i|\ast$'').
The maps $f_i$ and $g_i$ are associated with removing the only separator not adjacent to $i$. 
For instance, note that the compositions $f_0 \beta_0$ and $f_2 \alpha_2$ are equal since both maps remove the set of separators $\{0|1, 1|2\}$ from $0|1|2|$. We also have $f_0 \alpha_0 = f_1 \beta_1$, $f_1 \alpha_1= f_2 \beta_2$, $g_0 \beta_0 = g_1 \alpha_1$, $g_0 \alpha_0 = g_2 \beta_2$, $g_2 \alpha_2= g_1 \beta_1$.
\end{example}

\begin{lemma}
The face category $\FF(A_{n})$ is isomorphic to the category of cyclic partitions of $n$.
\end{lemma}
\begin{proof} 
Aguiar and Petersen describe the face structure of their notion of toric arrangement of Coxeter type $A_{n}$ (described in Example \ref{es:1}) in terms of "spin necklaces". Slightly paraphrasing \cite[§5.3]{AgPe}, the data of a spin necklace is a cyclically ordered partition of the set $[n+1]$ with a distinguished block that is labeled by an integer in $[n+1]$. Such necklaces describe orbits of faces of the reflection arrangement of type $\widetilde{A_{n}}$ under the action of the coroot lattice, via the explicit description given in \cite[p.\ 1166]{AgPe}. From there, we see that the action of the weight lattice on the arrangement of type $\widetilde{A_{n}}$ induces an action on spin necklaces that preserves the cyclic partition and the order among its blocks, but is transitive on the possible labelings. In particular, if the spin necklace has only one block and its label is $l$, then for every $k=1,\ldots n$ the $k$-th fundamental weight acts on this spin necklace by relabeling the block with $l-k {\mod n+1}$. 

The boundary of the face described by a given spin necklace is obtained by merging two adjacent blocks of the associated cyclic partition (the new distinguished block is the unique one that contains the original distinguished block and its label is obtained from the old by subtracting from it, modulo $n+1$, the cardinality of the unlabeled block that has been merged, if it was placed before the distinguished block in the cyclic order). Let us call $SN$ the resulting poset of Spin Necklaces.

In Example \ref{es:1} we see that $\FF(A_{{n}})$ is obtained by quotienting Aguiar and Petersen's poset $SN$ by the induced action of the weight lattice. This gives an acyclic category whose objects are plain (unlabeled) cyclic partitions of $n$. In order to determine orbits of morphisms, consider a covering relation $x\lessdot y$ in $SN$ and let $\pi(x)$, $\pi(y)$ be the cyclic partitions underlying the two spin necklaces. 

If both $\pi(x)$ and $\pi(y)$ have more than one block, then any labeling of $\pi(x)$ will induce a unique labeling of $\pi(y)$, and therefore there is only one orbit of morphisms $\pi(x)\to \pi(y)$ in $\FF(A_{{n}})$ which we can think of as ``removing the set of separators'' between adjacent blocks of $\pi(y)$ that are merged in $\pi(x)$.
Now suppose that $x$ has only one block, labelled with the number $l$.
Then $\pi(y)$ has two blocks, say $B_1$ and $B_2$. If $B_1$ is the distinguished block in $y$, then its label is determined as $l+\vert B_{2}\vert$, the sum taken modulo $n+1$.
Analogously, there is exactly one $y'\in SN$ with $\pi(y')=\pi(y)$ but with labeled block $B_2$.
Now, the only weight that fixes $x$, including labeling, is the zero weight.
Thus, the covering relation $x\lessdot y$ is in a different orbit than the covering relation $x\lessdot y'$.
Therefore, in  $\FF(A_{{n}})$ there are two different morphisms between $\pi(x)$ and $\pi(y)$ (in our "separators" parlance: if $B_1$ is labeled, we think of removing the separator $B_2|B_1$ in $B_1|B_2|$, otherwise the separator $B_1|B_2$).
Analogously one sees that for every order relation $x<z$ in $SN$, where $x$ has only one block and $z$ has $k$ blocks, in the quotient of $SN$ there are exactly $k$ distinct morphisms $\pi(x)\to \pi(z)$ -- which, in our representation, can be identified with the choice of $k-1$ block-separators to be removed from $\pi(z)$. 
\end{proof}

\begin{remark}\label{rem:found_group} 
The pure elliptic Braid group $P_{n+1}(\EC)$ is the fundamental group of the configuration space $F_{n+1}:= \set{p \in \EC^{n+1} \mid p_i \neq p_j \forall \, i,j \in [0,n]}$.
Since $F_{n+1}$ is isomorphic to the product $\EC \times M(A_n)$ via the morphism $p \mapsto (p_0, (p_i-p_0)_{i = 1, \dots, n})$, our presentation of $\pi_1(M(A_n))$ gives a presentation of $P_{n+1}(\EC)$.
To the authors' best knowledge, this presentation of the pure Braid group is different from the ones previously appeared in the literature (cf.\ \cite{Bellingeri04} \cite[Corollary 8]{GG04}). 
\end{remark}

\subsection{Fundamental group}
Let $\FF^n$ denote the set of chambers in $\FF(A_{n})$, and notice that $\FF^n$ is in bijection with the symmetric group $S_{n}$.
This bijection sends  
the permutation $\sigma = \left( \begin{smallmatrix}
1 & 2 & \dots & n \\
k_1 & k_2 & \dots & k_n
\end{smallmatrix} \right)$ to the chamber $C_\sigma = 0|k_1|k_2 \dots |k_n$. 
Let $\WW$ be the set of walls in $\FF(A_{n})$, each wall $W \in \WW$ is associated to a cyclic partition of $n$ into $n$ blocks.
The cardinality of $\WW$ is $\frac{1}{2}(n+1)!$.

Let $\WW^\circ \subset \WW$ be the set of all walls whose support is different from $\set{H_{0,i}}$ for all $i \in [1,n]$ (see \Cref{es:1}).
The simple reflections $s_i=(i,i+1) \in S_n$, for $i \in [1,n]$, form a Coxeter system for $S_n$.
The weak (right) order on the Coxeter group $S_n$ is given by the transitive closure of the following covering relation:
\[ \tau \vartriangleleft \sigma \Leftrightarrow \exists\, i \textnormal{ s.t.\ } \sigma=\tau s_i \textnormal{ and } l(\sigma)= l(\tau)+1 \]
The Hasse diagram of the weak order has $\FF^n$ as set of vertices and any edge between $\tau$ and $\sigma$ is associated to the wall (in $\WW^\circ$) between the chamber $C_\tau$ and $C_\sigma$.
For each $\sigma \in S_n$, $\sigma \neq \operatorname{id}$, we can consider the \textit{lexicographically first reduced word} $s_{i_1}s_{i_2} \dots s_{i_k}=\sigma$ and call $\tau$ the product $s_{i_1}s_{i_2} \dots s_{i_{k-1}}$.
Notice that $\tau \vartriangleleft \sigma$ and that $s_{i_1}s_{i_2} \dots s_{i_{k-1}}$ is the lexicographically first reduced word for $\tau$ by \cite[Theorem 3.4.8]{BjornerBrenti}.
As noticed before the covering relation $\tau \vartriangleleft \sigma$ defines a unique wall $W_\sigma$.
Let 
\begin{equation}\label{def:T}
T:=\{W_\sigma \mid \sigma \in S_n \setminus \set{\operatorname{id}}\}\subset \WW^\circ.
\end{equation}

\begin{lemma}
The set $T$ is a spanning tree of the graph $G_{A_n}$.
\end{lemma}

\begin{proof}
We will prove that the subgraph 
 induced by $T$ is connected.
Let $C_{\sigma} \in \FF^d(A_n)$ be an arbitrary vertex of $G_{A_n}$. We prove that it is connected to the vertex $C_{\operatorname{id}}=0|1|\dots |n|$ by induction on $l(\sigma)$. If $l(\sigma)=0$ then $\sigma=\operatorname{id}$ and there is nothing to show. 
Call $\tau$ the product $s_{i_1}s_{i_2} \dots s_{i_{k-1}}$, where $s_{i_1}s_{i_2} \dots s_{i_k}$ is the lexicographically first reduced word for $\sigma$.
By the inductive hypothesis we have that $C_\tau$ is connected to $C_{\id}$ since $l(\tau)=l(\sigma)-1$, moreover the wall $W_\sigma \in T$ connects $C_\tau$ with $C_\sigma$.\\
We complete the proof by noticing that $\lvert T \rvert = n!-1$ and the number of vertices of $G_{A_n}$ is $\lvert \FF^0(A_n) \rvert = n!$.
\end{proof}

Let $\PP_{2}\subseteq \FF^{d-2}$ be the set of $2$-codimensional faces of $K_{A_n}$ whose support is $\set{H_{i,j}, H_{j,k}}$ for some $i,j,k \in [0,n]$.
The other $2$-codimensional faces have support equal to $\set{H_{i,j}, H_{k,l}}$ for some pairwise different indexes $i,j,k,l \in [0,n]$ and let $\PP_{11}\subseteq \FF^{d-2}$ be the set of such faces.

We orient each edge in the graph $G_{A_n}$ as follows: for each wall $W= i_0 i_1|i_2| \dots |i_n|$ with $i_0<i_1$ the orientation is from the chamber $i_0|i_1|i_2| \dots |i_n|$ to the chamber $i_1|i_0|i_2| \dots |i_n|$.
We call the chamber $s(W)= i_0|i_1|i_2| \dots |i_n| \in \FF^n$ the \emph{source} and $t(W)= i_1|i_0|i_2| \dots |i_n| \in \FF^n$ the \emph{target} of the edge represented by $W$.
We also choose a reference chamber $C_{\operatorname{id}}=0|1|\dots |n|$.

\begin{theorem}
The fundamental group $\pi_1(M(A_{n,\EC}))$ of the complement of the elliptic braid arrangement is generated by $(C,W)$ and by $(W,C)$ for $C \in \FF^n$ and $W \in \WW$ with the following relations.
\begin{enumerate}
\item \label{rel:T1} $(W,C_{\operatorname{id}})=1$ for all $W\in T$,
\item \label{rel:T2} $(C,W)=1$ for all $W\in T$ and all $C\in \FF^n$,
\item \label{rel:pair_of_walls} $(t(V),W)(V,s(W))=(V,t(W))(s(V),W)$ for all $V,W \in \WW$ with $\supp V \neq \supp W$,
\item 
\label{rel:P2}
$(U,C)(V,C)(W,C)=(W',C)(V',C)(U',C)$ for all $C\in \FF^n$ and all $p=ijk|\dots | \in \PP_2$ ($i<j<k$), where
\begin{align*}
& U=k|ij|\dots| & & V=ik|j|\dots | & & W= i|jk|\dots | \\
& U'=ij|k|\dots| & & V'=j|ik|\dots | & & W'= jk|i|\dots |
\end{align*}
\item 
\label{rel:P2bis}
$(C,U)(C,V)(C,W)=(C,W')(C,V')(C,U')$ for all $C\in \FF^n$ and all $p \in \PP_2$ ($i<j<k$), with the notation of \eqref{rel:P2} above.
\item 
\label{rel:P11}
$(W,C)(V,C)=(V',C)(W',C)$ for all $C\in \FF^n$ and all $p= ij|\dots |kh|\dots | \in \PP_{1,1}$ ($i<j$ and $k<h$), where $V=i|j|\dots |kh|\dots |$, $V'=j|i|\dots |kh|\dots |$, $W= ij|\dots |h|k|\dots |$, and $W'= ij|\dots |k|h|\dots |$,
\item 
\label{rel:P11bis}
$(C,W)(C,V)=(C, V')(C,W')$ for all $C\in \FF^n$ and all $p\in \PP_{11}$, with the notation of \eqref{rel:P11} above.
\end{enumerate}
\end{theorem}

\begin{figure}
\begin{subfigure}[t]{0.49\textwidth}
\centering
\begin{tikzpicture}
    \node (p) at (0,3) {$\scriptstyle{i|j|k| \dots |}$};
    \node (e1) at (-1,1)  {$\scriptstyle{k|i|j| \dots |}$};
    \node (e3) at (1,1) {$\scriptstyle{j|k|i| \dots |}$};
    \node (g1) at (-1,2)  {$\scriptstyle{i|k|j| \dots |}$};
    \node (g3) at (1,2) {$\scriptstyle{j|i|k| \dots |}$};
    \node (f1) at (0,0) {$\scriptstyle{k|j|i| \dots |}$};
    \draw [->] (p) to node[left]{$W$} (g1);
    \draw [->] (g1) to node[left]{$V$} (e1);
    \draw [->] (g3) to node[right]{$V'$} (e3);
    \draw [->] (p) to  node[right]{$U'$} (g3);
    \draw [->] (e1) to node[left]{$U$} (f1);
    \draw [->] (e3) to node[right]{$W'$} (f1);
\end{tikzpicture}
\caption{The relation induced by $ijk|\dots | \in \PP_2$.} \label{fig:P2}
\end{subfigure}
\begin{subfigure}[t]{0.49\textwidth}
\centering
\begin{tikzpicture}
    \node (p) at (0,2) {$\scriptstyle{i|j|\dots |h|k|\dots|}$};
    \node (e1) at (-1,1)  {$\scriptstyle{j|i|\dots |h|k|\dots|}$};
    \node (e3) at (1,1) {$\scriptstyle{i|j|\dots |k|h|\dots|}$};
    \node (f1) at (0,0) {$\scriptstyle{j|i|\dots |k|h|\dots|}$};

    \draw [->] (p) to node[left]{$V$} (e1);
    \draw [->] (p) to node[right]{$W'$} (e3);
    \draw [->] (e1) to node[left]{$W$} (f1);
    \draw [->] (e3) to node[right]{$V'$} (f1);
\end{tikzpicture}
\caption{The relation induced by $ij|\dots |kh|\dots | \in\PP_{11}$.}
\label{fig:P11}
\end{subfigure}
\caption{The relations induced by the $2$-cells of the form $(C,p)$ and $(p,C)$.}
\end{figure}
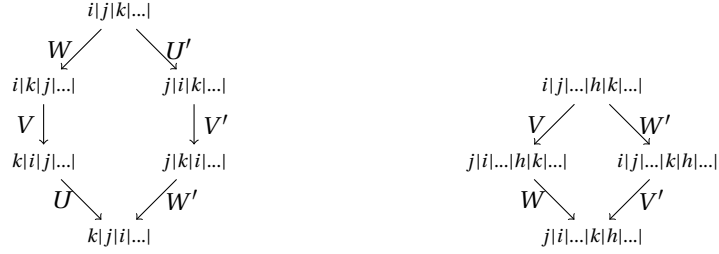

\begin{proof}
The result of \Cref{sect:fundamental_group} show that the fundamental group is generated by $(C,W)$ and by $(W,C)$ for $C \in \FF^n$ and $W \in \WW$, and that the relations are given by contracting a spanning tree $T_{\KK}$ and by all elements in $\KK^2$ (cf.\ \Cref{thm:pres}).
We choose as a spanning tree of $G_{A_n}$ the set $T$ from \Cref{def:T} and, by \Cref{lemma:tree_T_K}, the spanning tree $T_\KK$ of $\KK(A_n)$ is given by the left-hand side elements in \Cref{rel:T1,rel:T2} in the claim.
Recall from \Cref{def:KK} that the elements in $\KK^2(A_n)$ are of two kinds: pairs of walls and pairs consisting of a chamber and a $2$-codimensional face.
In the first case we obtain \Cref{rel:pair_of_walls}, the second case depends on the type of the $2$-codimensional face $p$:
if $p \in \PP_2$ then we obtain \Cref{rel:P2,rel:P2bis}, otherwise $p \in \PP_{11}$ and we obtain \Cref{rel:P11,rel:P11bis}.
\end{proof}

\appendix
\section{Polyhedral CW complexes and face categories}
\label{appendix}

\subsection{Acyclic categories and their nerves} \label{app:cat}
We will assume familiarity with the basic terminology of category theory and refer the reader to \cite{maclane} for clarification.
Given a 
 category $\CCC$ we will denote by $\Ob\CCC$, resp.\ $\Mor\CCC$ the sets of objects and morphisms of $\CCC$.
For every morphism $\phi\in \Mor\CCC$ we write $s(\phi)$ and $t(\phi)$ for the source, resp.\ the target of $\phi$.
Given objects $x,y\in \Ob\CCC$ we write $\Mor_{\CCC}(x,y)$ for the set of all morphisms $\phi$ with $s(\phi)=x$ and  $t(\phi)=y$.
Given $\phi_1, \phi_2\in \Mor\CCC$ with $t(\phi_1)=s(\phi_2)$ we write $\phi_1.\phi_2:= \phi_2 \circ \phi_1$ for the unique element of $\Mor(s(\phi_1),t(\phi_2))$ obtained as the concatenation (or composition) of $\phi_1$ and $\phi_2$.

An {\em acyclic category} (also called, e.g.\ in \cite{Bridson-Haefliger}, ``small category wihout loops'', or {\em scwol}) is any small category $\mathscr C$ such that the only invertible morphisms are endomorphisms, and the only endomorphisms are the identity morphisms.

\begin{remark}
Every partially ordered set $(P,\leq)$ can be thought of as an acyclic category $\CCC(P)$ with object set $P$ and such that, for every $x,y\in P$, $\Mor_{\mathscr C(P)}(x,y)$ is empty unless $x\leq y$, in which case it consists of exactly one element. 
\end{remark}

The {\em geometric realization} $\gr{\mathscr C}$ of an acyclic category $\mathscr C$ is defined as follows. Let $\Gamma_0=\Ob\CCC$ and, for all positive integers $d>0$, let $\Gamma_d$ denote the set of all ordered sets $\gamma=(\phi_1,\ldots,\phi_d)$ of non-identity composable morphisms of $\mathscr C$ (i.e., $t(\phi_{i})=s(\phi_{i+1})$ for all $i=1,\ldots d-1$). 
Moreover, for all $d>0$, every $\gamma\in \Gamma_d$ and $i=0,\ldots,d$ we define $\gamma_i\in \Gamma_{d-1}$ as follows:
\begin{align*}
\textrm{if } d=1, \,\,\,\, &\gamma_0:= s(\phi_1),\, \gamma_1:=t(\phi_1) \\
\textrm{if } d>1, \,\,\,\,
&\gamma_i:=\left\{
\begin{array}{ll}
(\phi_2,\ldots,\phi_d) & \textrm{if }i=0\\
(\phi_1,\ldots , \phi_{i-1},\phi_i.\phi_{i+1},\phi_{i+2},\ldots,\phi_d) & \textrm{if } 0<i<d \\
(\phi_1,\ldots,\phi_{d-1}) & \textrm{if }i=d
\end{array}\right.
\end{align*}

Now given $\gamma\in \Gamma_d$ define $\mathbb R^\gamma := \mathbb R^{s(\phi_1)}\oplus \mathbb R^{t(\phi_1)}  \oplus \mathbb R^{t(\phi_2)} \oplus \cdots \oplus \mathbb R^{t(\phi_d)}$, a $d+1$-dimensional Euclidean vector space with coordinates indexed by sources and targets of morphisms in $\gamma$.
Then, let $\Delta[\gamma]$ be the standard $d$-simplex in $ \mathbb R^\gamma$, i.e., the set of all $v\in \mathbb R^\gamma_{\geq 0}$ with sum of all coordinates equal to $1$, equipped with the subspace topology.
Moreover, for all $i=0,\ldots,d$ we write $\partial_\gamma^i:\Delta[\gamma_i]\to \Delta[\gamma]$ for the restriction of the natural inclusion $\mathbb R^{\gamma_i}\subseteq \mathbb R^{\gamma}$.

This defines a diagram of topological spaces $\Delta$ indexed over the set $\Gamma:=\bigcup_i\Gamma_i$ with partial order given by inclusion. Then we define
\[
\gr{\mathscr C}:=
\colim_{\Gamma}
\Delta = 
\biguplus_{\gamma\in \Gamma} \Delta[\gamma] 
\left/
\left(\begin{array}{c}
\partial^i_\gamma (x)\sim x\\
\textrm{for all }\gamma\in \Gamma_i,\, x\in\Delta[\gamma_i]  
\end{array}
\right)\right.
\]

Notice that, as a set, $\gr{\mathscr C}$ consists of the disjoint union $\biguplus_{\gamma\in \Gamma}\interior{\Delta[\gamma]}$ of the relative interiors of all simplices $\Delta[\gamma]$. In particular, the canonical 
continuous maps
\[
\alpha_\gamma: \Delta[\gamma] \rightarrow \gr{\mathscr C}
\]
are injective on the interior of $\Delta[\gamma]$.

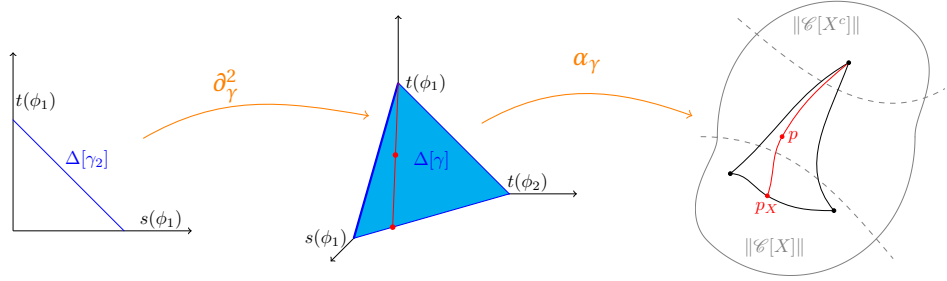
\begin{figure}[ht]
\begin{center}
\begin{tikzpicture}
\node (L) at (-4.5,0) {\includegraphics[scale=.7]{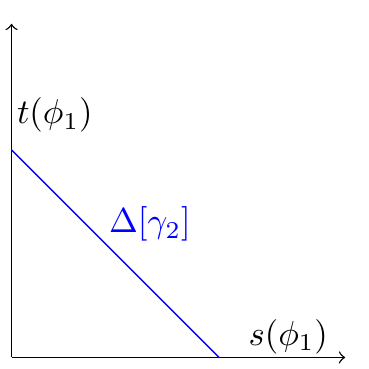}};
\node (C) at (0,0) {\includegraphics[scale=.7]{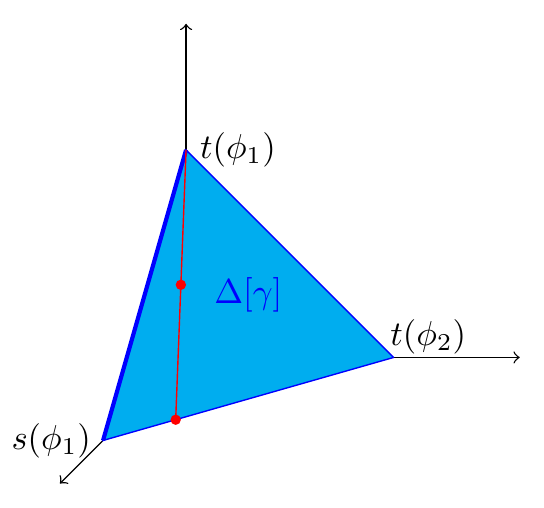}};
\node (R) at (5,0) {\includegraphics[scale=.7]{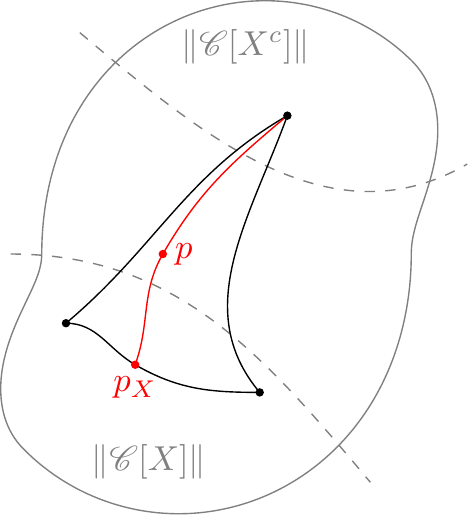}};
\draw[orange,->] (-4,0) to[out=30,in=170] (-1,.3);
\node[orange] (d) at (-2.9,.7) {$\partial_\gamma^2$};
\draw[orange,->] (.5,0.2) to[out=30,in=160] (3.3,.3);
\node[orange] (d) at (1.9,1) {$\alpha_\gamma$};
\end{tikzpicture}
\end{center}
\caption{Geometric illustration of the maps $\partial_\gamma^2$ and $\alpha_\gamma$ for $\gamma=(\phi_1,\phi_2)$, and of the retraction described in the proof of \Cref{lem:retract}}
\end{figure}

For any given small category $\mathscr C$ and any subset $X\subseteq \Ob(\mathscr C)$, let us denote by $\mathscr C [X]$ the {\em full subcategory} determined by $X$. In particular, this implies that $\Ob(\mathscr C[X])=X$ and $\Mor_{\mathscr C[X]}(x,y)=\Mor_{\mathscr C}(x,y)$ for all $x,y\in X$, see \cite[§ I.3, p.\ 15]{maclane}.

\begin{lemma}\label{lem:retract}
Let $\CCC$ be an acyclic category, choose $X\subseteq \Ob(\CCC)$ and consider the full subcategories $\CCC[X]$ and $\CCC[X^c]$ on the set $X$, resp.\ on its complement $X^c:=\Ob\CCC \setminus X$.
Then the space $\gr{\CCC} \setminus \gr{\CCC[X^c]}$ retracts onto the subcomplex $\gr{\CCC[X]}\subseteq \gr{\CCC}$.
\end{lemma}

\begin{proof}
The subcomplex $\gr{\CCC[X]}$ of $\gr{\CCC}$ consists exactly of the cells  whose vertices all lie in $X$. 
Now recall that for every point $p\in \gr{\CCC}$
there is a unique chain of morphisms $\gamma$ such that $p$ is in the open cell $\alpha_{\gamma}(\interior{\Delta[\gamma]})$. Now the vertices of $\Delta[\gamma]$ can be partitioned in two sets  $A$ and $B$, according to whether the vertex's image under $\alpha_\gamma$ lies in $\gr{\CCC[X]}$ (then the vertex is in $A$) or not, and we can set $\Delta_A:=\operatorname{conv} A$, $\Delta_B:=\operatorname{conv} B$.
Then, $\Delta[\gamma]=\Delta_A\ast \Delta_B$. \\
 In particular, if $p\not\in \gr{\mathscr C[X^c]} \cup \gr{\mathscr C[X]}$ then $A,B\neq \emptyset$  
 and every point in the interior of $\Delta[\gamma]$ lies in the interior of a line segment between a unique $p'_A\in \Delta_A$   and a unique $p'_B\in \Delta_B$. In this case, let $p_X:=\alpha_\gamma(p'_A)$ be the point corresponding to $p'_A\in \Delta_A$.
If $p\not\in \gr{\mathscr C[X]}$, let $p_X:=p$. Then the map
\[
f\colon \gr{\CCC} \setminus \gr{\CCC[X^c]}\to \gr{\CCC[X]}, \,\,
p\mapsto p_X
\]
is a deformation retraction (the homotopy to the identity on $\gr{\CCC[X]}$ can be constructed, e.g., as in the proof of \cite[Lemma 70.1]{Munkres}).
\end{proof}

\begin{definition}\label{def:slice}
Given a category $\CCC$ and an object $x\in \Ob(\CCC)$, write $\CCC / x$ for the slice category, i.e., the category whose objects are all morphisms of $\CCC$ ending in $x$ and where, for any two given   $\phi \colon a\to x$ and $\psi : b\to x$ in $\Mor_{\CCC}$, there is one morphism $\mu_\zeta: \phi\to\psi$ for every $\zeta\in \Mor_{\CCC}(a,b).$
Composition of morphisms is given by $\mu_{\zeta}.\mu_{\zeta'} = \mu_{\zeta.\zeta'}$ whenever $\zeta.\zeta'$ are composable in $ \Mor_{\CCC}$. 
We will denote by $\overline{\CCC/x}$ the full subcategory of $\CCC/ x$ on the object set $\Ob(\CCC)\setminus \{\id_x\}$. 
\end{definition}

 Notice that, topologically, $\gr{{\CCC/x}}$  is a cone over $\gr{\overline{\CCC/x}}$. There is a natural functor
\[
J_x \colon {\CCC/x}\to \CCC
\]
that sends every $\phi \in \Ob{\CCC/x}$ to $J_x(\phi)=s(\phi)$, the source object of $\phi$ and every morphism $\mu_{\zeta}\in \Mor_{\CCC/x}$ to the morphism $\zeta\in\Mor_{\CCC}$ that corresponds to it. This functor induces a continuous function 
\begin{equation}\label{jx}
j_x \colon \gr{{\CCC/x}}\to \gr{\CCC}
\end{equation}

\begin{lemma}
The map $j_x$ is injective in the interior of each simplex of $\gr{{\CCC/x}}$ that has $\id_x$ as a vertex.
\end{lemma}
\begin{proof}
The identity $\id_x$ is a vertex of the simplex $\Delta[\gamma]$ with $\gamma=(\mu_{\zeta_1}, \dots, \mu_{\zeta_d})$ if and only if $t(\mu_{\zeta_d})=\id_x$, i.e., $t(\zeta_d)=x$.
Now, $J_x$ maps every sequence $\gamma=(\mu_{\zeta_1},\ldots,\mu_{\zeta_d})$ of composable morphisms in ${\CCC/x}$ with $t(\mu_{\zeta_d})=\id_x$ to the sequence $(\zeta_1,\ldots,\zeta_d)$ of composable morphisms of $\CCC$, with $t(\zeta_d)=x$. The evident injectivity of this map implies injectivity of $j_x$ on the open cell $\alpha_\gamma(\interior{\Delta[\gamma]})\subseteq \gr{{\CCC/x}}$.
\end{proof}

\subsection{Polyhedral CW complexes}

Recall that a polytope $P$ is any convex hull of a finite set of points in Euclidean space. Equivalently, a polytope is any bounded subset of $\mathbb R^d$ obtained as the intersection of finitely many halfspaces \cite[Theorem 1.1]{Ziegler}. A face of $P$ is any non-empty subset of the form $F=\{x\in P \mid \ell(x)=c\}$, where $\ell$ is any linear form and $c\in \mathbb R$ is such that $\ell(x)\geq c$ for all $x\in P$. We call {\em polyhedral homeomorphism} any homeomorphism $f:P\to Q$ such that the image of every face of $P$ is a face of $Q$.
Recall that the existence of such a polyhedral homeomorphism is equivalent to the fact that the posets of faces of $P$ and $Q$ are isomorphic, see \cite[§ 2.2]{Ziegler}.

\begin{definition}
A {\em polyhedral CW complex} is a Hausdorff space $K$ together with a family of continuous maps $\{\alpha:P_\alpha\to K\}_{\alpha\in I}$ where each $P_\alpha$ is a convex polytope and, writing $F(\alpha):=\alpha(\interior{P_\alpha})$, the following conditions hold.
\begin{itemize}
\item[(PC0)] The restriction $\alpha\vert_{\interior{P_\alpha}}$ is a homeomorphism  for all $\alpha\in I$.
\item[(PC1)] $K=\biguplus_{\alpha\in I} F(\alpha)$. 
\item[(PC2)] Any $A\subseteq K$ is closed in $K$ if and only if $A\cap \alpha(P_\alpha)$ is closed for every $\alpha\in I$.
\item[(PC3)] For every $\alpha\in I$ and every face $F$ of $P_\alpha$, there is $\beta\in I$ and a polyhedral homeomorphism $\phi^F: P_\beta \to F$ such that $\beta = \alpha \circ \phi^F$.  
\end{itemize}
The {\em face category} of the polyhedral complex $K$ is the category $\mathscr F(K)$ with
\begin{itemize}
\item $\Ob(\mathscr F(K))=I$
\item The set of morphisms of $\mathscr F(K)$ is the set of all $\phi^F:P_\beta\to F$ where $F$ ranges over the faces of all $P_\alpha$, $\alpha\in I$ ($\beta$ is then uniquely determined, see (PC3)). Sources and targets are defined by $s(\phi^F)=\beta$ and $t(\phi^F)=\alpha$.
The composition of $\phi^F\in\Mor_{\mathscr F(K)} (\alpha,\beta)$ with $\phi^G\in\Mor_{\mathscr F(K)} (\beta,\gamma)$ is defined as 
$$\phi^F.\phi^G=
\phi^{\phi^G(F)}\in \Mor_{\mathscr F(K)} (\alpha,\gamma).$$
\end{itemize}
\end{definition}

\begin{example}
Let $K=S^1$, we define a polyhedral CW complex on $K$ as follows.
Let $\alpha \colon q \to S^1$ defined by $\alpha(q)= 1 \in S^1$, where $q$ is the $0$-dimensional polyhedron, \ie a point.
Let $\beta \colon [0,1] \to S^1$ be the map defined by $\beta(t) = e^{2\pi i t} \in S^1$, where $[0,1]=P_\beta$ is a $1$-dimensional polyhedron.
The set $\set{\alpha,\beta}$ defines a polyhedral CW complex homeomorphic to $S^1$, see \Cref{fig:poly_complex}.
The proper faces of $P_\beta$ are $\set{0}$ and $\set{1}$, the four maps $\phi^F$ are $\phi^{\set{1}} \colon q \to \set{1}$, $\phi^{\set{0}} \colon q \to \set{0}$, $\phi^{P_\alpha} = \id_{P_\alpha}$, and $\phi^{P_\beta} = \id_{P_\beta}$.
\begin{figure}
\centering
\begin{subfigure}[t]{0.49\textwidth}
\centering
\begin{tikzpicture}
\draw [fill] (0,0) circle [radius=0.05];
\node [below] at (0,0) {$F(\alpha)=\beta(0)=\beta(1)$};
\draw [thick] (0,0.5) circle [radius=0.5];
\node [above] at (0,1) {$F(\beta)$};
\end{tikzpicture}
\subcaption{The polyhedral CW complex $K\cong S^1$ consisting of one edge $P_\beta=[0,1]$ whose endpoints are identified in the single vertex $P_\alpha=q$.}
\end{subfigure}
\begin{subfigure}[t]{0.49\textwidth}
\centering
\begin{tikzpicture}
    \node (p) at (0,0) {$P_\alpha$};
    \node (e) at (0,1.5)  {$P_\beta$};
    \node at (-0.5,0.5){$\phi_\alpha^{0}$};
    \node at (0.5,0.5){$\phi_\alpha^{1}$};
    \node at (0.6,2.2){$\phi_\beta^{P_\beta}$};
    \node at (0.6,-0.7){$\phi_\alpha^{P_\alpha}$};
    \draw [->,bend left=20] (p) to (e);
    \draw [->,bend right=20] (p) to (e);
    \draw [->] (e.70) arc (-60:240:.7em);
    \draw [->] (p.-70) arc (60:-240:.7em);
\end{tikzpicture}
\subcaption{The category $\FF(K)$.}
\end{subfigure}
\caption{A polyhedral CW complex and its face category.}
\label{fig:poly_complex}
\end{figure}
\end{example}

\begin{remark}\label{rem:unique}$\,$
\begin{itemize}
\item[(1)] The maps $\alpha$ are open and proper maps (as is every continuous function from a compact space to a Hausdorff space).
\item[(2)]
Notice that if the identity $x.\phi^F = \phi^{G}$ in $\Mor_{\mathscr F(K)}$ has a solution then this solution is unique and given by $\phi^{(\phi^F)^{-1}(G)}$. 
\end{itemize}
\end{remark}

\begin{example}
Every polytope $P$ is itself a polyhedral complex, with respect to the family $\{F\hookrightarrow P\}_{F\textrm{ face of }P}$ of inclusions of all nonempty faces. In this case the face category $\mathscr F(P)$ is a poset, called the {\em poset of faces} of $P$. 
\end{example}
\begin{remark}
If $P$ and $Q$ are two polytopes, then the posets $\mathscr F(P)$ and $\mathscr F(Q)$ are isomorphic if and only if there exists a polyhedral homeomorphism $P\to Q$ \cite[§ 2.2, p. 58]{Ziegler}.
\end{remark}

\begin{example}
Our definition of a polyhedral CW complex encompasses the notion of a polyhedral complex in the sense of \cite[Section 2.2.4]{Kozlov} (where regularity is required) as well as the polyhedral complexes considered in \cite[Chapter 7]{Bridson-Haefliger} (where the ``face maps'' $\phi^F$ are required to be isometries).
\end{example}

\begin{example}
The nerve of every acyclic category is a polyhedral CW complex where all polytopes $P_\alpha$ are simplices.
\end{example}

\subsection{Polyhedral CW complexes and acyclic categories}
The fundamental relationship between a polyhedral complex and its face category is expressed by the following proposition.
\begin{proposition}
Let $K$ be a polyhedral CW complex. Then there is a homeomorphism $\gr{\mathscr F(K)}\cong K$.
\end{proposition}
\begin{proof}
Define a function $f:\gr{\mathscr F(K)}\to K$ as follows. Given $p\in \gr{\mathscr F(K)}$ we consider the unique chain $\gamma$ of composable morphisms in $\mathscr F(K)$ such that the image  $F(\gamma)$ of $\alpha_\gamma: \interior{\Delta[\gamma]}\to \gr{\mathscr F(K)}$ is the unique open face of the complex $\gr{\mathscr F(K)}$ that contains $p$. We can let $\alpha$ be the target of the composition of the morphisms in $\gamma$. Then, there is a canonical inclusion $\iota : \Delta[\gamma]\hookrightarrow P_\alpha$ as an open cell of the barycentric subdivision. Now let
$$
f(x):=\alpha\circ\iota\circ(\alpha_\gamma)^{-1}(x).
$$
since $\alpha$ is an open map (Remark \ref{rem:unique}.(1)) and $\alpha_\gamma$ is a homeomorphism on the interior of $\Delta$, $f$ is continuous. Moreover, $f$ is clearly bijective and it is proper (preimages of compacta are compact). By \cite[Proposition 2.1]{Walker}, $f$ is a homeomorphism.
\end{proof}

\begin{lemma}\label{lem:sphere}
Let $K$ be a polyhedral CW complex. For every $\alpha\in \Ob\mathscr F(K)$, the slice category ${\mathscr F(K)/\alpha}$ is isomorphic to the poset of faces of the polytope $P_\alpha$. 
\end{lemma}
\begin{proof} That $\mathscr F(K)/\alpha$ is a poset follows from Remark~\ref{rem:unique}, and it has finitely many objects by definition of $K$. Thus we only need to consider the map
$
f: \Ob({\mathscr F(K)/\alpha}) \to \mathscr F(P_\alpha)$, 
$
\phi^F \mapsto F
$
and prove the following claims
\begin{itemize}
\item $f$ is bijective. Surjectivity follows from (PC3) above. 
For injectivity notice that $f(\phi^F)=f(\phi^{G})$ implies $F=G$ and thus $\phi^F=\phi^{G}$ (recall (PC1) and (PC3)).
\item 
If there is a morphism $\phi^F\to \phi^{G}$ in $\mathscr F(K)/\alpha$, then $F\subseteq G$. In fact, if such a morphism exists then there is a morphism $\phi^H$ of $\FF(K)$ such that $\phi^H.\phi^{G}=\phi^F$. the definition of the composition of morphisms in $\FF(K)$ now gives $\phi^{\phi^{G}(H)}=\phi^F$ and in particular $\phi^{G}(H)=F$. This means that $F\subseteq \im(\phi^G)=G$ as desired. \qedhere
\end{itemize}
\end{proof}

\begin{lemma}\label{lem:CW}
Let $\CCC$ be an acyclic category such that for every $x\in \Ob(\CCC)$ the category ${\CCC/x}$ is the face category of a finite-dimensional polytope.
Then, there is a structure of a polyhedral CW complex $K$ on $\gr{\CCC}$ such that $\CCC$ is the face category of $K$. 
Moreover, for every $x\in \Ob\CCC$ the cell $F(x)$ of $K$ indexed by $x$ is the image of $\interior{\gr{\CCC/x}}$ under $j_x$.
\end{lemma}

\begin{proof}
For every object $x\in \Ob\mathscr C$ we can choose a polytope $P_x$ such that ${\CCC/x} \simeq \mathscr F(P_x)$ and call  $b_x:P_x \to \gr{{\CCC/x}}$ the PL-homeomorphism  given by barycentric subdivision, see e.g.\ \cite[§ 2.2]{Kozlov}. \\ 
Let $K:=\gr{\mathscr C}$ and, for every object $x\in \Ob\mathscr C$, recall the continuous map from Equation \eqref{jx} and set
$$
\alpha_x:= j_x\circ b_x:\quad P_x\to K.
$$
We check that the family of maps $\alpha_x$ define a polyhedral CW complex structure on $K$. Items (PC0), (PC1), (PC2) are clear.
For item (PC3) consider any $x\in \Ob\mathscr C$ and let $F$ be a face of $P_x$. Let $\beta\in \Ob {\CCC/x}$ be the object corresponding to $F\in\mathscr F(P_x)$ and let $y:=s(\beta)\in \Ob\CCC$. Notice that the assignment $\chi\mapsto \chi.\beta$ gives a functor ${\CCC/y} \to {\CCC/x}$ and hence a continuous function $i_{yx} \colon \gr{{\CCC/y}}\to \gr{{\CCC/x}}$ that is injective because of Remark \ref{rem:unique}.(2).
Moreover, $i_{yx}$ is closed as a map into $\gr{{\CCC/x}}$ since  finiteness -- hence compactness -- of the complex $\gr{{\CCC/y}}$ ensures that $i_{yx}$ is a homeomorphism onto its (closed) image.\\
Then, the composition
\[
(b_x)^{-1}\circ\iota_{yx}\circ b_y \colon P_y \to P_x
\]
has image in $F$.
This map gives the desired polyhedral homeomorphism.
\end{proof}

The following can be seen as a counterpart for polyhedral CW complexes of Bj\"orner's characterization of poset of cells of regular CW-complexes from \cite{BjornerCW}. 

\begin{theorem}
An acyclic category $\mathscr C$ is isomorphic to the face category of a polyhedral CW complex if and only if every slice ${\mathscr{C}/x}$ is isomorphic to the face poset of a polytope.
\end{theorem}
\begin{proof}
The claim follows from \Cref{lem:sphere,lem:CW}.
\end{proof}

\subsection{Group actions on acyclic categories}
Let $\mathscr C$ be an acyclic category and let $G$ be a group acting on $\mathscr C$. Thus to every $g\in G$  is associated a functor $\mathscr G(g):\mathscr C\to \mathscr C$ such that $\mathscr G(g) \circ \mathscr G(h) = \mathscr G(gh)$ for all $g\in G$ and $\mathscr G(\id_G)$ is the identity functor $\Id_{\CCC}$. In particular, the functor $\mathscr G(g)$ must be an isomorphism of categories, for all $g$. This is usually formalized as a functor $\mathscr G: G\to \AC$ from the group $G$ viewed as a category with one object \cite[I.2]{maclane} to the category $\AC$ of acyclic categories that sends the unique object of $G$ to $\mathscr C$. 
The categorical quotient $\mathscr C/G$ is defined as the colimit of $\mathscr G$ in $\AC$. Below we show that, under some mild conditions, the quotient category coincides with the category whose set of objects and morphisms coincide with the sets of orbits of object and morphisms, respectively, with the natural composition. This is the case, for instance, when $\mathscr C$ is {ranked}.

\begin{definition}
An indecomposable morphism in an acyclic category is one that cannot be factored in a nontrivial composition of morphisms. Call an acyclic category {\em ranked} if there is a function $\rho :\CCC \to \mathbb N$ with $\rho^{-1}(0)\neq \emptyset$ and such that, for every indecomposable, non-identity morphism $x\to y$, $\rho(y)=\rho(x)+1$.
\end{definition}

\begin{example}
If $K$ is a polyhedral CW complex, then the category $\mathscr F(K)$ is ranked by dimension of the cells (\ie, $\rho(\alpha)=\dim P_\alpha$ for all $\alpha\in I$).
\end{example}

\begin{lemma}
If a group acts on a ranked acyclic category $\CCC$, the quotient acyclic category $\CCC/G$ has object set $\Ob(\CCC/G)=(\Ob\CCC)/G$, the set of orbits of objects. Morphisms of $\CCC/G$ are orbits of morphisms of $\CCC$. Given a morphism $\phi:x\to y$ in $\Mor\CCC$, its $G$-orbit is a morphism $G\phi: Gx\to Gy$.  Composition of morphisms is given by $G\phi . G\psi = G(\phi.\psi)$ whenever $\phi$ and $\psi$ are composable.
\end{lemma}
\begin{proof}
Since any self-isomorphism of a ranked category is rank-preserving, the diagram $\mathscr G$ is ``geometric'' in the sense of \cite[Definition 5.2]{DD2}, hence \cite[Proposition 5.6]{DD2} applies.
\end{proof}

\begin{lemma}\label{lem:cover_cat}
If an action of a group $G$ on a ranked acyclic category $\CCC$ is free on $\Ob\CCC$, then the natural functor $\CCC\to \CCC/G$ is a covering of categories and induces a covering map $\gr{\CCC}\to \gr{\CCC/G}$ with monodromy group $G$.\\
In particular, for every $x\in \Ob\CCC$, the slices $\CCC/x$ and $(\CCC/G)/Gx$ are isomorphic.
\end{lemma}
\begin{proof}
In \cite[III.$\refc C$.11.1]{Bridson-Haefliger} two conditions are given that, together with freeness of the action on objects, ensure $\CCC\to\CCC/G$ to be a covering of acyclic categories in the sense of \cite[III.$\refc C$.1.9]{Bridson-Haefliger}.
The first of these conditions is satisfied because the group $G$ acts by rank-preserving functors.
The second condition holds because, if $G$ acts freely on $\Ob\mathscr C$, then for every morphism $\phi$ of $\CCC$ we have that whenever  $\mathscr G(g)$ fixes $s(\phi)$, then $g=\id_G$ and hence in particular $\mathscr G (g) (\phi)= \phi$. 

The claim now follows with \cite[III.$\refc C$.1.9.(1)]{Bridson-Haefliger}, where it is noted that if a functor is a covering of categories then it induces a (topological) covering map of their geometric realizations.
The statement about isomorphism of slices holds because the quotient functor is a covering, see \cite[p.\ 532 and Definition A.15 on p.\ 579]{Bridson-Haefliger}.
\end{proof}

\begin{corollary}
If an action of a group $G$ on a ranked acyclic category $\CCC$ is free on $\Ob\CCC$, then there is a natural isomorphism of complexes $\gr{\CCC}/G\simeq \gr{\CCC/G}$. 
\end{corollary}
\begin{proof}
Kozlov gives a sufficient condition on a group action in order for there to be an equivalence as stated in the claim, see \cite[14.3.2]{Kozlov}.
Now, \cite[Lemma 4.12]{DD1} ensures that this condition is satisfied whenever the group action is a covering in the sense of \cite[III.$\refc C$.1.13]{Bridson-Haefliger}.
\end{proof}

\bibliographystyle{amsalpha}
\bibliography{ellittici}

\bigskip\bigskip

\end{document}